\documentclass[a4paper,12pt]{article}
\usepackage[latin1]{inputenc}
\usepackage[francais]{babel}
\usepackage[T1]{fontenc}
\usepackage{latexsym}
\usepackage{amsmath,amsthm}
\usepackage{amsfonts}
\usepackage{epsfig}
\usepackage{amssymb}
\usepackage{amscd}
\usepackage[all]{xy}
\usepackage{stmaryrd}
\usepackage{rotating}
\input xy

\usepackage{hyperref}

\newtheorem{Lemma}{Lemme}[section]
\newtheorem{Theo}[Lemma]{Théorème}
\newtheorem{Cor}[Lemma]{Corollaire}
\newtheorem{Pro}[Lemma]{Proposition}

\theoremstyle{definition}
\newtheorem{Def}[Lemma]{Définition}

\newtheorem{Theoreme}[]{Théorème}

\theoremstyle{definition}

\theoremstyle{remark}

\theoremstyle{remark}
\newtheorem{Rem}[Lemma]{Remarque}

\newcommand{\NN}{\mathbb{N}}
\newcommand{\RR}{\mathbb{R}}
\newcommand{\CC}{\mathbb{C}}

\newcommand{\G}{\rtimes^{\rho} G}
\newcommand{\rG}{\rtimes^{\rho}_{r} G}

\newcommand{\End}{\mathrm{End}}
\numberwithin{equation}{section}

\newcommand{\A}{\mathcal{A}^{\rho}_r(G)}

\newcommand{\Amax}{\mathcal{A}^{\rho}(G)}

\newcommand{\EEnd}{\mathrm{End}}

\title{Représentations non unitaires, morphisme de Baum-Connes et complétions inconditionnelles}

\author{Maria Paula Gomez Aparicio}

\long\def\symbolfootnote[#1]#2{\begingroup%
\def\thefootnote{\fnsymbol{footnote}}\footnote[#1]{#2}\endgroup}

\date{}

\begin{document}
\maketitle
\begin{abstract}
Nous montrons la bijectivité du morphisme de Baum-Connes tordu par une représentation non unitaire, défini dans \cite{Gomez08}, pour une grande classe de groupes vérifiant la conjecture de Baum-Connes et qui contient en particulier tous les groupes de Lie réel semi-simples, tous les groupes hyperboliques, ainsi que des groupes inifinis discrets ayant la propriété (T) de Kazhdan. Nous définisons une opération de tensorisation par une représentation de dimension finie non unitaire sur le membre de gauche du morphisme de Baum-Connes et nous montrons que son analogue en $K$-théorie est forcément défini sur la $K$-théorie des algèbres de groupe tordues introduites dans \cite{Gomez07}.\\

\begin{center}{\bf Abstract}
 
\end{center}

 \indent We show that the Baum-Connes morphism twisted by a non-unitary representation, defined in \cite{Gomez08}, is an isomorphism for a large class of groups satisfying the Baum-Connes conjecture. Such class contains all the real semi-simple Lie groups, all hyperbolic groups and many infinite discret groups having Kazhdan's property (T). We define a tensorisation by a non-unitary finite dimensional representation on the left handside of the Baum-Connes morphism and we show that its analogue in $K$-theory must be defined on the $K$-theory of the twisted group algebras introduced in \cite{Gomez07}.

\end{abstract}



\symbolfootnote[0]{\phantom{a}\hspace{-7mm}\textit{2000 MSC.}22D12, 22D15, 46L80, 19K35}

\symbolfootnote[0]{\phantom{a}\hspace{-7mm}\textit{Keywords and
    phrases:} Non-unitary representations, Banach algebras, $KK$-theory, Baum-Connes conjecture}


\section*{Introduction}
Soit $G$ un groupe localement compact et $\rho$ une représentation de $G$, pas nécessairement unitaire, sur un espace de dimension finie $V$. Dans \cite{Gomez07} (voir aussi \cite{Gomez08} et \cite{GomezThese}), nous avons défini un analogue \emph{tordu par $\rho$} de la $C^*$-algèbre réduite de $G$, que nous avons noté $\mathcal{A}_r^{\rho}(G)$, en considérant la complétion de l'espace vectoriel $C_c(G)$ des fonctions continues à support compact sur $G$ pour la norme donnée par la formule
$$\|f\|_{\mathcal{A}_r^{\rho}(G)}=\|(\lambda_{G}\otimes\rho)(f)\|_{\mathcal{L}(L^{2}(G)\otimes V)},$$
pour tout $f\in C_c(G)$ et où $\lambda_{G}$ est le représentation régulière gauche de $G$. Nous avons ensuite défini un morphisme 
$$\mu_{\rho,r}:K^{\mathrm{top}}(G)\rightarrow K(\A),$$
qui coïncide avec le morphisme de Baum-Connes $\mu_r$ (cf. \cite{Baum-Connes-Higson}), si $\rho$ est unitaire. Nous avons alors montré que $\mu_{\rho,r}$ est injectif pour tout groupe $G$ pour lequel il existe un élément $\gamma$ de Kasparov (cf. \cite[Théorème 3.8]{Gomez08} et \cite{Tu99} pour la définition d'un élément $\gamma$). Nous avons en plus montré que si cet élément $\gamma$ est égal à $1$ dans $KK_G(\CC,\CC)$, alors $\mu_{\rho,r}$ est un isomorphisme (cf. \cite[Théorème 3.7]{Gomez08}). L'existence d'un élément $\gamma$ de Kasparov a été montrée pour une classe très large de groupes, notée $\mathcal{C}$ dans \cite{Lafforgue02} (voir aussi \cite{Gomez08} et \cite{GomezThese}), qui contient en particulier tous les groupes de Lie semi-simples, ainsi que tous les sous-groupes fermés d'un groupe de Lie semi-simple. Cette classe contient donc, en particulier, des groupes ayant la propriété (T) de Kazhdan (cf. \cite{Kazhdan}, \cite{delaHarpe-Valette}). Cependant, dès qu'un groupe localement compact $G$ a la propriété (T), la représentation triviale étant isolée dans le dual unitaire de $G$, il est impossible de construire une homotopie entre un élément $\gamma$ de Kasparov et $1$ dans $KK_G(\CC,\CC)$. Ceci implique que les résultats de \cite{Gomez08} ne sont pas suffissants pour montrer la surjectivité de $\mu_{\rho,r}$ pour des groupes ayant la propriété (T). Dans cet article, nous allons montrer que le morphisme de Baum-Connes tordu $\mu_{\rho,r}$ est un isomorphisme pour beaucoup de groupes contenus dans la classe $\mathcal{C}$ et ayant la propriété (T). Plus précisément, nous allons montrer le théorème suivant
\begin{Theoreme}
 Soit $G$ un groupe localement compact satisfaisant les deux conditions suivantes:
 \begin{enumerate}
 \item il existe une complétion inconditionnelle de $C_c(G)$ qui est une sous-algèbre de $C^*_r(G)$ dense et stable par calcul fonctionnel holomorphe (ou plus précisément faiblement pleine dans $C^*_r(G)$, voir définition \ref{fplein});
  \item pour toute complétion inconditionnelle $\mathcal{B}(G)$ de $C_c(G)$ le morphisme $\mu_{\mathcal{B}}: K^{\mathrm{top}}(G)\rightarrow K(\mathcal{B}(G))$ est un isomorphisme.
      \end{enumerate}
 Soit $\rho$ une représentation de dimension finie de $G$. Alors le morphisme de Baum-Connes tordu par $\rho$ $$\mu_{\rho,r}:K^{\mathrm{top}}(G)\rightarrow K(\mathcal{A}^{\rho}_r(G)),$$
est un isomorphisme.
\end{Theoreme}

On rappelle que Lafforgue dans \cite[Théorème 0.02]{Lafforgue02} a démontré que, pour toute complétion inconditionnelle $\mathcal{B}(G)$, le morphisme $$\mu_{\mathcal{B}}:K^{\mathrm{top}}(G)\to K(\mathcal{B}(G)),$$ est bijectif pour une classe de groupes qu'il note $\mathcal{C'}$. On rappelle ici que la classe $\mathcal{C'}$ est constituée par les groupes a-T-menables et par tous les groupes localement compacts agissant de façon continue, isométrique et propre sur un des espaces suivants:\\
\begin{itemize}
 \item sur une variété riemannienne complète simplement connexe, dont la courbure sectionnelle est négative ou nulle et bornée inférieurement, et dont la dérivée du tenseur de courbure (suivant la connexion induite de la connexion de Levi-Civita sur le fibré tangent) est bornée, ou
\item sur un immeuble affine de Bruhat-Tits uniformément localement fini, ou
\item sur un espace métrique uniformément localement fini, faiblement géodésique, faiblement bolique et vérifiant une condition supplémentaire de bolicité forte.\\
\end{itemize}
Cette classe contient, en particulier, tous les groupes de Lie semi-simples réels, ainsi que tous les sous-groupes fermés d'un groupe de Lie semi-simple réel.\\

D'autre part, si $G$ est groupe de Lie semi-simple réel, Lafforgue a constuit une variante de l'algèbre de Schwartz généralisée $\mathcal{S}(G)$ (cf. \cite{Lafforgue02}), qui est une sous-algèbre dense et stable par calcul fonctionnel holomorphe dans $C^*_r(G)$. Ceci implique que le morphisme $$K(\mathcal{S}(G))\to K(C^*_r(G)),$$ induit par l'inclusion, est un isomorphisme, et donc que le morphisme $\mu_{\rho,r}$ est aussi un isomorphisme. De plus, $\mu_{\rho,r}$ est un isomorphisme pour tout groupe de la classe $\mathcal{C}'$ ayant la propriété (RD); dans ce cas c'est une variante de l'algèbre de Jolissaint associée à $G$ \cite{Jolissaint} qui est une complétion inconditionnelle, dense et stable par calcul fonctionnel holomorphe dans $C^*_r(G)$. En particulier, $\mu_{\rho,r}$ est un isomorphisme pour tout groupe hyperbolique et pour tous les sous-groupes discrets et cocompacts de $SL_3(F)$, où $F$ est un corps local, de $SL_3(\mathbb{H})$ et de $E_{6(-26)}$ (cf. \cite{Lafforgue00}, \cite{Chatterji03}, \cite{delaHarpe-Valette}).\\

Les résultats de cet article complètent les résultats de \cite{Gomez08}. Le morphisme de Baum-Connes tordu est alors un isomorphisme pour la plupart des groupes pour lesquels on sait montrer la conjecture de Baum-Connes et on peut même espérer que les deux soient vérifiés toujours au même temps. Nous espérons ainsi que le morphisme de Baum-Connes tordu par n'importe quelle représentation de dimension finie, soit un isomorphisme au moins pour tout groupe appartenant à la classe $\mathcal{C}$.\\

D'autre part, dans \cite{Valette88}, Valette a défini une action de l'anneau des représentations non unitaires de dimension finie de $G$, que nous notons $R_F(G)$, sur le membre de gauche du morphisme de Baum-Connes, $K^{\mathrm{top}}(G)$, qui généralise l'action induite par produit tensoriel par une représentation de dimension finie dans le cas des groupes compacts. Dans le cas des groupes de Lie connexes, il définit une action de $R_F(G)$ sur l'image de l'élément $\gamma$ de Kasparov dans $K(C^*_r(G))$ qu'il interprète ensuite en termes de foncteurs de Zuckerman. En général, pour toute représentation de dimension finie $(\rho,V)$ de $G$, on aimerait définir un endomorphisme de $K(C^{*}_r(G))$ qui soit analogue au morphisme induit par produit tensoriel par $\rho$ sur l'anneau des représentations de $G$. Ceci définirait un espèce d'analogue en $K$-théorie des foncteurs de translation de Zuckerman \footnote{Dans le cadre de la $K$-théorie les foncteurs de Zuckerman ont été considérés par Bost.} (cf. \cite{Zuckerman77}, \cite[Chapter VII]{Knapp-Vogan}). Cependant, comme $\rho$ n'est pas supposée unitaire, le produit tensoriel par $\rho$ induit un morphisme de $C_c(G)$ dans $C^*_r(G)\otimes\mathrm{End}(V)$ qui n'est pas défini sur $C^*_r(G)$ et dont le domaine de définition est l'algèbre tordue $\A$. On a donc un morphisme d'algèbres de Banach de $\A$ dans $C^*_r(G)\otimes\mathrm{End}(V)$, qui, par équivalence de Morita, induit un morphisme en $K$-théorie
$$\Lambda_{\rho}:K(\A)\rightarrow K(C^*_r(G)).$$

De plus, l'action de $R_F(G)$ sur $K^{\mathrm{top}}(G)$ est définie de façon très naturelle. Dans le cas des représentations unitaires, cette action a été définie par Kasparov (cf. \cite{Kasparov88}). Dans le cas général, pour toute représentation de dimension finie $\rho$ de $G$ nous allons définir un endomorphisme $$\Upsilon_{\rho}:K^{\mathrm{top}}(G)\rightarrow  K^{\mathrm{top}}(G),$$ qui coïncide avec le produit tensoriel par $\rho$ lorsque est unitaire. L'action de $R_F(G)$ sur $K^{\mathrm{top}}(G)$ induite par $\Upsilon_{\rho}$ que nous allons alors définir coïncide avec l'action définie par Valette dans \cite{Valette88}.

Le but de la deuxième partie de cet article est de montrer le théorème suivant
\begin{Theoreme}\label{Ktopprincipal}
 Soit $G$ un groupe localement compact et $\rho$ une représentation non unitaire de dimension finie de $G$. Alors le diagramme suivant
 $$\xymatrix{
    K^{\mathrm{top}}(G)\ar[r]^{\mu_{\rho,r}}\ar[d]_{\Upsilon_{\rho}}& K(\A)\ar[d]^{\Lambda_{\rho}}\\
K^{\mathrm{top}}(G)\ar[r]^{\mu_r}& K(C^*_r(G)),
}$$
est commutatif.\\
\end{Theoreme}

\medskip
\noindent{\bf Remerciements.} Les résultats de cet article, ainsi que ceux de \cite{Gomez08}, font partie des travaux présentés pour l'obtention de mon doctorat réalisé sous la direction de Vincent Lafforgue. Je voudrais le remercier de m'avoir suggéré ce sujet, ainsi que pour tous ses conseils. Je remercie aussi Alain Valette pour m'avoir indiqué \cite{Valette88} et pour ses commentaires.

\section{Rappels et notations}
Soit $G$ un groupe localement compact et soit $dg$ une mesure de Haar à gauche sur $G$. On note $\Delta$ la fonction modulaire de $G$, c'est-à-dire que $dg^{-1}=\Delta(g)^{-1}$ pour tout $g\in G$.\\
On appelle longueur sur $G$ toute fonction $\ell:G\to [0,+\infty[$ continue et telle que $\ell(g_1g_2)\leq \ell(g_1)+\ell(g_2),$ pour tout $g_1, g_2 \in G$.\\
Étant donnée une $G$-$C^*$-algèbre $A$, pour tout $g\in G$ et $a\in A$, on note $g.a$, ou $g(a)$, l'action de $g$ sur $a$. On considère l'espace vectoriel des fonctions continues à support compact sur $G$ à valeurs dans $A$, que l'on note $C_c(G,A)$, muni de la structure d'algèbre involutive dont la multiplication est donnée par la formule
$$(f_1*f_2)(g)=\int_Gf_1(g_1)g_1(f_2(g_1^{-1}g))dg_1,$$ pour $f_1,f_2\in C_c(G,A)$ et l'involution par la formule
$$f^*(g)=g(f(g^{-1}))^*\Delta(g^{-1}),$$ pour $f\in C_c(G,A)$, $g\in G$. Intuitivement, on représente tout élément $f\in C_c(G,A)$ par l'intégrale formelle $\int_Gf(g)e_gdg$, où $e_g$ est une lettre formelle qui satisfait les conditions suivantes $$e_ge_{g'}=e_{gg'}, e_g^*=(e_g)^{-1}=e_{g^{-1}}\quad\text{et}\quad e_gae_g^{-1}=g.a,$$ pour tout $g\in G$ et pour tout $a\in A$.\\

Soient $A$ et $B$ des $G$-$C^*$-algèbres. Pour toute longueur $\ell$ sur $G$, on note $\iota$ le morphisme
$$\iota:KK_G(A,B)\to KK^{\mathrm{ban}}_{G,\ell}(A,B),$$ défini dans \cite[Proposition 1.6.1]{Lafforgue02}.

Si $A$, $B$ et $D$ sont des $G$-algèbres de Banach, on note $A\otimes^{\pi} D$ le produit tensoriel projectif, c'est-à-dire le complété-séparé du produit tensorile algébrique $A\otimes^{alg}_{\CC}D$ pour la plus grande semi-norme $\|.\|$ telle que $\|a\otimes d\|\leq \|a\|_A\,\|d\|_D$, pour tout $a\in A$ et $d\in D$. On note alors 
$\sigma_D$ le morphisme $$\sigma_D:KK^{\mathrm{ban}}_{G,\ell}(A,B)\to KK^{\mathrm{ban}}_{G,\ell}(A\otimes^{\pi}D,B\otimes^{\pi}D),$$
défini dans \cite[Définition 1.2.2]{Lafforgue02}.\\
De plus, on note $\Sigma$ le morphisme $$\Sigma: KK^{\mathrm{ban}}(A,B)\to\mathrm{Hom}\big(K(A),K(B)\big),$$
défini dans \cite[Proposition 1.2.9]{Lafforgue02} et qui induit l'action de la $KK$-banachique sur la $K$-théorie.\\

Soit $\rho$ une représentation de $G$ sur un espace vectoriel complexe $V$ de dimension finie et muni d'une structure hermitienne. Dans \cite{Gomez08}, nous avons défini des produits croisés tordus par la représentation $\rho$ en considérant l'application suivante
\begin{align*}
 C_c(G,A)&\rightarrow C_c(G,A)\otimes\End(V)\\
\int_Gf(g)e_gdg&\mapsto\int_Gf(g)e_g\otimes\rho(g)dg.
\end{align*}
On rappelle la défintion des produits croisés tordus
\begin{Def}\label{produitscroisés}
Le produit croisé (resp. produit croisé réduit) de $A$ et $G$ tordu par la représentation $\rho$, noté $A\rtimes^{\rho} G$ (resp. $A\rtimes^{\rho}_r G$), est le complété (séparé) de $C_c(G,A)$
pour la norme:
$$\|\int_Ga(g)e_gdg\|_{A\G}=\|\int_Ga(g)e_g\otimes\rho(g)dg\|_{C^*(G,A)\otimes\End(V)},$$
(resp. $\|.\|_{C^*_r(G,A)\otimes\End(V)}$) où $C^*(G,A)\otimes \End(V)$ (resp.$C^*_r(G,A)\otimes\End(V)$) est le produit tensoriel minimal de $C^*$-algèbres.\\
Si $A=\CC$, alors on note $$\mathcal{A}^{\rho}(G):=\CC\rtimes^{\rho}G\quad\hbox{et}\quad\mathcal{A}^{\rho}_r(G):=\CC\rtimes_r^{\rho}G.$$ 
\end{Def}

\begin{Rem}
 L'algèbre $\A$ coïncide avec le complété-séparé de $C_c(G)$ par la norme $$\|f\|=\|(\lambda_G\otimes\rho)(f)\|_{\mathcal{L}(L^2(G)\otimes V)},$$
où $\lambda_G$ est la représentation régulière gauche de $G$. De même, pour tout $f\in C_c(G)$, $$\|f\|_{\Amax}=\sup\limits_{(\pi,H)}\|(\pi\otimes\rho)(f)\|_{\mathcal{L}(H\otimes V)},$$ où $(\pi,H)$ parcourt l'ensemble des représentations unitaires de $G$.\\
\end{Rem}

Soient $A$ et $B$ des $G$-$C^*$-algèbres. On note $j_{\rho}$ et $j_{\rho,r}$ les morphismes de descente \emph{tordus} 
$$j_{\rho}:KK_G(A,B)\to KK^{\mathrm{ban}}(A\G, B\G),$$
$$j_{\rho,r}:KK_G(A,B)\to KK^{\mathrm{ban}}(A\rG, B\rG),$$
définis dans \cite[Définition 2.9]{Gomez08}.\\  

Considérons maintenant $\underline{E}G$ l'espace classifiant pour les actions propres de $G$. Pour toute $G$-$C^*$-algèbre $A$, on note $K^{\mathrm{top}}(G,A)$ la $K$-homologie $G$-équivariante de $\underline{E}G$ à valeurs dans $A$ introduite dans \cite{Baum-Connes-Higson}, c'est-à-dire $$K^{\mathrm{top}}(G,A)=\lim\limits_{\longrightarrow}KK_G(C_0(X),A),$$
où la limite inductive est prise parmi les $G$-espace propres $X$ qui sont des sous-espaces fermés de $\underline{E}G$ tels que $X/G$ est compact.\\
Dans \cite[Définition 2.17]{Gomez08}, nous avons défini deux morphismes de groupes
$$\mu^A_{\rho}:K^{\mathrm{top}}(G,A)\rightarrow K(A\G)\quad\hbox{et}\quad\mu^A_{\rho,r}:K^{\mathrm{top}}(G,A)\rightarrow K(A\rG),$$de la façon suivante: \\
Soit $X$ est un espace $G$-propre et $G$-compact. Soit $c$ une fonction continue à support compact sur $X$ à valeurs dans $\RR^+$ telle que, pour tout $x\in X$, $\int_G c(g^{-1}x)dg=1$ (on rappelle qu'une fonction avec ces propriétés existe parce que $X/G$ est compact, voir \cite{Tu99}). On considère la fonction sur $G\times X$
$$p(g,x):=\sqrt{c(x)c(g^{-1}x)},$$
qui définit un projecteur de $C_c(G,C_0(X))$. On note $p_{\rho}$ (resp. $p_{\rho,r}$) l'élément de $K(C_0(X)\G)$ (resp. $K(C_0(X)\rG)$) qu'il définit.\\ 
Soit $x\in KK_G(C_0(X),A)$. Alors $\mu^A_{\rho}$ et $\mu^A_{\rho,r}$ sont donnés, à passage à la limite inductive près, par les formules $$\mu_{\rho,r}^A(x)=\Sigma(j_{\rho,r}(x))(p_{\rho,r})\quad\text{et}\quad\mu_{\rho,r}^A(x)=\Sigma(j_{\rho,r}(x))(p_{\rho,r}),$$
où, $$\Sigma:KK^{\mathrm{ban}}\big(C_0(X)\G, A\G\big)\to \mathrm{Hom}\Big(K\big(C_0(X)\G\big),K\big(A\G\big)\Big),$$ 
est donné par l'action de $KK^{\mathrm{ban}}$ sur la $K$-théorie; de même pour les produits croisés tordus réduits.\\

 Si $A=\CC$, on note $\mu_{\rho}$ (resp. $\mu_{\rho,r}$) le morphisme $\mu^{\CC}_{\rho}$ (resp. $\mu^{\CC}_{\rho,r}$) pour simplifier les notations. Nous allons montrer que $\mu_{\rho}$ est un isomorphisme pour une large classe de groupes vérifiant la conjecture de Baum-Connes et qui contient des groupes ayant la propriété (T).\\

Tout le long de cet article, un groupe localement compact $G$ sera toujours dénombrable à l'infini. Une représentation $\rho$ de dimension finie de $G$ sera une représentation de $G$ sur un espace vectoriel complexe de dimension finie muni d'une norme hermitienne. On note $(\rho, V)$ toute représentation de $G$ sur un espace $V$, pour simplifier les notations et quand on veut préciser l'espace sur lequel $G$ agit. Si $A$ est une $C^*$-algèbre, on note $\widetilde{A}$ son algèbre unitarisée.

\section{Morphisme de Baum-Connes tordu}\label{complétinc}
\subsection{Complétions inconditionnelles}
On rappelle la définition de complétion inconditionnelle introduite dans \cite{Lafforgue02}.
\begin{Def}\label{complétionincond}
 Soit $G$ un groupe localement compact. Une algèbre de Banach $\mathcal{B}(G)$ est une complétion inconditionnelle de $C_c(G)$, si elle contient $C_c(G)$ comme sous-algèbre dense et si, quels que soient $f_1,f_2\in C_c(G)$ tels que $|f_1(g)|\leq|f_2(g)|$ pour tout $g\in G$, on a $\|f_1\|_{\mathcal{B}(G)}\leq\|f_2\|_{\mathcal{B}(G)}$, autrement dit, si pour tout $f\in C_c(G)$, $\|f\|_{\mathcal{B}(G)}$ ne dépend que de $(g\mapsto|f(g)|)$.
\end{Def}
\begin{Def}\label{fplein}
 Soit $B$ une algèbre de Banach et $C$ une sous-algèbre dense. Soit $A$ une sous-algèbre de $B$ qui est une complétion de $C$ pour une norme telle que $\|x\|_{B}\leq\|x\|_{A}$ pour tout $x\in C$. L'algèbre $A$ est une sous-algèbre faiblement pleine de $B$ relativement à $C$, si pour tout $n\in\NN^*$ et pour tout $x\in \mathrm{M}_n(C)$,
$$\rho_{\mathrm{M}_n(A)}(x)=\rho_{\mathrm{M}_n(B)}(x),$$
où $\rho$ denote le rayon spectral.
\end{Def}
\begin{Rem}
Par définition une sous-algèbre faiblement pleine d'une algèbre de Banach est dense.
\end{Rem}
L'intérêt de cette notion est donné par la proposition suivante démontrée dans \cite[Lemme 1.7.2]{Lafforgue02}
\begin{Pro}\label{Kthsurjectif}
 Soient $A, B, C$ comme dans la définition \ref{fplein}. Notons $\theta:C\rightarrow B$ et $\theta_1:C\rightarrow A$ les inclusions évidentes. Soit $\tau$ l'unique morphisme d'algèbres de Banach de $A$ dans $B$ tel que $\tau\circ\theta_1=\theta$. Si $A$ est une sous-algèbre faiblement pleine de $B$ alors le morphisme induit par $\tau$ en $K$-théorie
$$\tau_*:K(A)\rightarrow K(B),$$
est surjectif.
\end{Pro}
\begin{Rem}
 La notion de \emph{faiblement pleine} est un peu plus faible que le fait d'être dense et stable par calcul fonctionnel holomorphe, mais, grâce à la proposition \ref{Kthsurjectif}, elle est suffisante pour nos propos. On peut comparer cette notion avec la notion de sous-algèbre dense et \emph{pleine} définie dans \cite[Définition 4.4.5]{Lafforgue02}. Voir aussi le théorème A.2.1 et la proposition A.2.2 de \cite{Bost90}.
\end{Rem}

On veut montrer le théorème suivant
\begin{Theo}
 Soit $G$ un groupe localement compact et $(\rho,V)$ une représentation de dimension finie de $G$. Soit $C_c(G)$ l'espace des fonctions continues à support compact sur $G$. S'il existe une sous-algèbre faiblement pleine de $C^*_r(G)$ qui est une complétion inconditionnelle de $C_c(G)$, alors il existe une sous-algèbre faiblement pleine de $\mathcal{A}^{\rho}_r(G)$ qui est aussi une complétion inconditionnelle de $C_c(G)$.
\end{Theo}

Soient $G$ un groupe localement compact et $(\rho,V)$ une représentation de dimension finie de $G$. Supposons qu'il existe $\mathcal{B}(G)$ une complétion inconditionnelle de $C_c(G)$ qui est une sous-algèbre faiblement pleine de $C_r^*(G)$ et notons $\mathfrak{i}$ l'inclusion.\\
On note $\mathcal{B}(G,\mathrm{End}(V))$ la complétion de $C_c(G,\mathrm{End}(V))$ pour la norme
$$\|f\|_{\mathcal{B}(G,\mathrm{End}(V))}=\|g\mapsto\|f(g)\|_{\mathrm{End}(V)}\|_{\mathcal{B}(G)},$$
pour $f\in C_c(G,\mathrm{End}(V))$.\\
D'après \cite[Proposition 1.6.4]{Lafforgue02}, il existe un morphisme d'algèbres de Banach de $\mathcal{B}(G,\mathrm{End}(V))$ dans $C^*_r(G,\mathrm{End}(V))$ prolongeant $\mathrm{Id}_{C_c(G,\EEnd(V))}$. On note $\tau$ ce morphisme.
\begin{Lemma}\label{Morita}
L'algèbre de Banach $\mathcal{B}(G,\mathrm{End}(V))$ est une sous-algèbre faiblement pleine de $C^*_r(G,\mathrm{End}(V))$.
\end{Lemma}

\begin{proof}
Soit $m=\mathrm{dim}_{\CC}(V)$. Comme $\mathcal{B}(G)$ est faiblement pleine dans $C^*_r(G)$, on a
$$\rho_{\mathrm{M}_n(\mathcal{B}(G))}(x)=\rho_{\mathrm{M}_n(C^*_r(G))}(x),$$
pour tout $n\in\NN^*$ et pour tout $x\in \mathrm{M}_n(C_c(G))$. Donc, pour tout $k\in\NN^*$, et pour tout $x\in\mathrm{M}_{km}(C_c(G))$,
$$\rho_{\mathrm{M}_{k}(\mathrm{M}_{m}(\mathcal{B}(G)))}(x)=\rho_{\mathrm{M}_{k}(\mathrm{M}_{m}(C^*_r(G)))}(x).$$
Or, on a des isomorphismes d'algèbres de Banach à équivalence de norme près (resp. de $C^*$-algèbres)
\begin{align*}
\mathcal{B}(G,\EEnd(V))\simeq\mathrm{M}_m(\mathcal{B}(G))\quad\hbox{et}&\quad C^*_r(G,\EEnd(V))\simeq\mathrm{M}_m(C^*_r(G)),
\end{align*}
et donc ceci implique que $\mathcal{B}(G,\EEnd(V))$ est faiblement pleine dans $C^*_r(G,\EEnd(V))$.
\end{proof}

Considérons maintenant la complétion de $C_c(G)$ pour la norme donnée par:
$$\|f\|_{\mathcal{B}^{\rho}(G)}=\|g\mapsto|f(g)|\|\rho(g)\|_{\mathrm{End}(V)}\|_{\mathcal{B}(G)},$$
pour $f\in C_c(G)$. On note cette algèbre $\mathcal{B}^{\rho}(G)$ et on remarque que c'est une complétion inconditionnelle.

\begin{Theo}\label{ssalpondérée}
 L'algèbre $\mathcal{B}^{\rho}(G)$ est une sous-algèbre faiblement pleine de $\mathcal{A}^{\rho}_r(G)$.
\end{Theo}

\begin{proof}
On considère l'application suivante:
\begin{align*}
C_c(G)&\rightarrow C_c(G,\EEnd(V))\\
f&\mapsto (g\mapsto f(g)\rho(g)),
\end{align*}
et on note $\tau_{1}$ le morphisme d'algèbres de Banach de $\mathcal{B}^{\rho}(G)$ dans $\mathcal{B}(G,\mathrm{End}(V))$ qu'elle induit.\\
Soit $\tau_2:\mathcal{A}_r^{\rho}(G)\rightarrow C^*_r(G)\otimes\EEnd(V)$ le morphisme isométrique induit par l'application $f\mapsto\int_G f(g)e_g\otimes\rho(g)dg$.
On a le diagramme commutatif suivant:

$$\xymatrix{
    \mathcal{B}(G,\EEnd(V))\ar[r]_{\tau} & C^*_r(G,\EEnd(V))\ar[r]^{\simeq}& C^*_r(G)\otimes\EEnd(V)\\
\mathcal{B}^{\rho}(G)\ar[rr]_{\psi}\ar[u]_{\tau_1}& & \mathcal{A}_r^{\rho}(G)\ar[u]_{\tau_2},
}$$
où $\psi:\mathcal{B}^{\rho}(G)\rightarrow \mathcal{A}_r^{\rho}(G)$ est l'unique morphisme continu d'algèbres de Banach prolongeant $\mathrm{Id}_{C_c(G)}$.
On veut montrer que $\psi$ est un morphisme faiblement plein.\\

Le morphisme $\tau_1:\mathcal{B}^{\rho}(G)\rightarrow\mathcal{B}(G,\EEnd(V))$ est isométrique, donc pour tout pour tout $n\in\NN^*$ et $x\in\mathrm{M_n}(C_c(G))$,
 \begin{align*}
  \rho_{\mathrm{M_n}(\mathcal{B}^{\rho}(G))}(x)&=\rho_{\mathrm{M_n}(\mathcal{B}(G,\EEnd(V)))}(\mathrm{M_n}(\tau_1)(x)).
\end{align*}
De plus, comme $\mathcal{B}(G)$ est pleine dans $C^*_r(G)$, le lemme \ref{Morita} implique,
\begin{align*}
\rho_{\mathrm{M_n}(\mathcal{B}(G,\EEnd(V)))}(\mathrm{M_n}(\tau_1)(x))&=\rho_{\mathrm{M_n}(C^*_r(G,\EEnd(V))}(\mathrm{M_n}(\tau\circ\tau_1)(x))
\end{align*}
et par commutativité du diagramme on a
$$\rho_{\mathrm{M_n}(C^*_r(G,\EEnd(V)))}(\mathrm{M_n}(\tau\circ\tau_1)(x))=\rho_{\mathrm{M_n}(C^*_r(G)\otimes\EEnd(V))}(\mathrm{M_n}(\tau_2\circ\psi)(x)).$$
Or, comme $\tau_2$ est isométrique le membre de droite de la dernière égalité est égal à $\rho_{\mathrm{M_n}(\mathcal{A}_r^{\rho}(G))}(\mathrm{M_n}(\psi)(x))$, d'où l'égalité
\begin{align*}
\rho_{\mathrm{M_n}(\mathcal{B}^{\rho}(G))}(x)&=\rho_{\mathrm{M_n}(\mathcal{A}_r^{\rho}(G))}(\mathrm{M_n}(\psi)(x)),\\
&=\rho_{\mathrm{M_n}(\mathcal{A}_r^{\rho}(G))}(x)
\end{align*}
la dernière égalité provenant du fait que $x\in\mathrm{M_n}(C_c(G))$ et que $\mathrm{M_n}(\psi)$ prolonge $\mathrm{Id}_{\mathrm{M_n}(C_c(G))}$.
\end{proof}

\subsection{Applications à Baum-Connes tordu}
On donne maintenant une proposition qui nous permettra d'appliquer les résultats précédents à l'étude de la bijectivité du morphisme de Baum-Connes tordu réduit. Étant donné une $G$-$C^*$-algèbre $B$ et une compétion inconditionnelle $\mathcal{B}(G)$, on note $\mathcal{B}^{\rho}(G,B)$ la complétion de $C_c(G,B)$ pour la norme
$$\|f\|_{\scriptstyle\mathcal{B}^{\rho}(G,B)}=\|g\mapsto\|f(g)\|_B\|\rho(g)\|_{\mathrm{End}(V)}\|_{\mathcal{B}(G)},$$
pour $f\in C_c(G,B)$. 
\begin{Pro}\label{BCcomplétions}
Pour toute $G$-$C^*$-algèbre $B$, on a le diagramme commutatif suivant
$$\xymatrix {K^{\mathrm{top}}(G,B)\ar[r]^{\mu^B_{\mathcal{B}^{\rho}}}\ar[dr]^{\mu^B_{\rho,r}}& K(\mathcal{B}^{\rho}(G,B))\ar[d]^{\psi_{B,*}}\\
&K(B\rG)},$$
où $\mu^B_{\mathcal{B}^{\rho}}$ est la variante du morphisme de Baum-Connes à coefficients définie dans \cite[Section 1.7.1]{Lafforgue02} pour les complétions inconditionnelles, et $\psi_B:\mathcal{B}^{\rho}(G,B)\rightarrow \mathcal{A}_r^{\rho}(G,B)$ est l'unique morphisme d'algèbres de Banach prolongeant l'identité.
\end{Pro}
\begin{proof}
 La démonstration est analogue à celle de la proposition 1.7.6 de \cite{Lafforgue02}. Soit $A$ une $G$-$C^*$-algèbre et soit $$\psi_A:\mathcal{B}^{\rho}(G,A)\rightarrow A\rG,$$ prolongeant l'identité sur $C_c(G,A)$. On rappelle que l'on note $\iota$ l'application $$\iota:KK_G(A,B)\rightarrow KK_G^{\mathrm{ban}}(A,B),$$définie par Lafforgue (cf. \cite[Proposition 1.6.1]{Lafforgue02}).\\
On a alors que, pour tout élément $\alpha\in KK_G(A,B)$, $$\psi_{B,*}(j_{\mathcal{B}^{\rho}}(\iota(\alpha)))=\psi^*_A(j^{\rho}_r(\alpha)),$$ dans $KK^{\mathrm{ban}}(\mathcal{B}^{\rho}(G,A),B\rG)$.\\
 En effet, on construit une homotopie entre les éléments $\psi_{B,*}(j_{\mathcal{\rho}}(\iota(\alpha)))$ et $\psi^*_A(j^{\rho}_r(\alpha))$ de $E^{\mathrm{ban}}(\mathcal{B}^{\rho}(G,A),B\rG)$, à l'aide de cônes (voir \cite{Lafforgue02}, \cite{Gomez08}). Soit $(E,T)$ un représentant de $\alpha$ dans $KK_G(A,B)$. On considère l'application
\begin{align*}
C_c(G,E)\otimes C_c(G,B)&\to C_c(G,E) \\
x\otimes b&\mapsto x.b.
\end{align*}
Grâce au lemme 1.6.6 de \cite{Lafforgue02}, on montre facilement qu'elle peut être prolongée en un morphisme de $B\rG$-modules de Banach à droite
\begin{align*}
 \tau:\mathcal{B}^{\rho}(G,E)^>\otimes^{\pi}_{\widetilde{\mathcal{B}^{\rho}(G)}}\widetilde{B\rG}&\to (E\rG)^>,
\end{align*}
qui est de norme inférieure ou égale à $1$. On pose $$\mathcal{C}(\tau)^>:=\{(h,x)\in(E\rG)[0,1]\times\psi_{B,*}\big(j_{\mathcal{B}^{\rho}}(\iota(\alpha))\big)^>\,|\, h(0)=\tau(x)\},$$
le cône associé à ce morphisme. De même, on définit $\mathcal{C}(\tau)^>$, comme le cône associé au morphisme
$$\overline{\tau}:\widetilde{B\rG}\otimes^{\pi}_{\widetilde{\mathcal{B}^{\rho}(G)}}\mathcal{B}^{\rho}(G,E)^<\to (E\rG)^<.$$
Soit $\mathcal{C}(\tau,T)$, l'opérateur sur $\mathcal{C}(\tau)$ défini par
$$\mathcal{C}(\tau,T)^>(h,e\otimes b)=\Big((g,t)\mapsto T(h(t)(g)), \big(g\mapsto T(e(g))\big)\otimes b\Big),$$
pour $h\in (E\rG)^>[0,1]$ et $x=e\otimes b\in \psi_{B,*}\big(j_{\mathcal{B}^{\rho}}(\iota(\alpha))\big)$; de même pour $\mathcal{C}(\tau,T)^<$.\\
On a alors que $(\mathcal{C}(\tau),\mathcal{C}(\tau,T))$ appartient à $E^{\mathrm{ban}}\big(\mathcal{B}^{\rho}(G,A),(B\rG)[0,1]\big)$. En effet, soit $\Phi$ l'application
\begin{align*}
 \Phi:C_c(G,\mathcal{K}(E))&\to\mathcal{L}(\mathcal{C}(\tau))\\
S&\mapsto \big((h,x)\to(t\mapsto \widehat{S_{\rho}}(h(t)),\psi_{B,*}(\widehat{S})x\big),
\end{align*}
où $\widehat{S_{\rho}}$ et $\widehat{S}$ sont les élément de $\mathcal{L}_{B\rG}(E\rG)$ et de $\mathcal{L}_{\mathcal{B}^{\rho}(G)}(\mathcal{B}^{\rho}(G,E))$ définis à partir de $S\in C_c(G,\mathcal{K}(E)$ dans \cite[Définition 2.5]{Gomez08} et \cite[Lemme 1.5.6]{Lafforgue02}, respectivement. On montre alors que l'image de $\Phi$ est contenue dans les opérateurs compacts de $\mathcal{C}(\tau)$, en remarquant les deux faits suivants
\begin{enumerate}
\item Si $E$ est un $G$-$(A,B)$-bimodule de Banach et $S=(S_g)_{g\in G}$ appartient à $C_c(G,\mathcal{K}(E))$, alors
\begin{align*}
\|\widehat{S_{\rho}}\|_{\mathcal{L}(E\rtimes^{\rho}_rG)}&\leq\|g\mapsto\|S_g\|_{\mathcal{K}(E)}\|\rho(g)\|_{\mathrm{End}(V)}\|_{L^1(G)} \\
&=\|g\mapsto\|S_g\|_{\mathcal{K}(E)}\|_{L^{1,\rho}(G)},
\end{align*}
où on note $L^{1,\rho}(G)$ la complétion de $C_c(G)$ pour la norme $L^1$ pondérée
$$\|f\|_{L^{1,\rho}(G)}=\int_G|f(g)|\|\rho(g)\|_{\mathrm{End}(V)}dg,$$
qui est une complétion inconditionnelle de $C_c(G)$.\\
\item Dans le lemme 1.5.6 de \cite{Lafforgue02}, on peut choisir les $y_i$ et les $\xi_i$ tels que
$$\|g\mapsto\|S_g-S_{0,g}\|_{\mathcal{K}(E)}\|_{\mathcal{B}^{\rho}(G)}+\|g\mapsto\|S_g-S_{0,g}\|_{\mathcal{K}(E)}\|_{L^{1,\rho}(G)}\leq\epsilon.$$
\end{enumerate}
Il est ensuite facil de voir que, pour $a\in \mathcal{B}^{\rho}(G,A)$ et $g\in G$, 
\begin{align*}
 \Phi(S_1)&=[a,\mathcal{C}(\tau,T)],\\
\Phi(S_2)&=a\big(1-(\mathcal{C}(\tau,T))^2\big),\\
\text{et}\quad\Phi(S_3)&=a\big(g(\mathcal{C}(\tau,T))-\mathcal{C}(\tau,T)\big),
\end{align*}
où pour $a\in C_c(G,A)$ et $g\in G$, 
\begin{align*}
&S_1:=\big(t\mapsto a(t)(t(T)-T)+[a(t),T]\big),\\
&S_2:=\big(t\mapsto a(t)t(1-T^{2})\big),\\
\hbox{et}\quad&S_3:=\big(t\mapsto a(t)t((gT)-T)\big),
\end{align*}
de sorte que $S_i\in C_c(G,\mathcal{K}(E))$ pour $i=1,..,3$. On a donc que
$$[a,\mathcal{C}(\tau,T)],\quad a\big(1-(\mathcal{C}(\tau,T))^2\big)\quad\text{et}\quad a\big(g(\mathcal{C}(\tau,T))-\mathcal{C}(\tau,T)\big),$$
appartiennent à l'image de $\Phi$ (voir aussi \cite[Lemme 2.7 et Lemme 2.14]{Gomez08}). Il est clair alors que $\mathcal{C}(\tau,T)$ réalise une homotopie entre les éléments $\psi_{B,*}(j_{\mathcal{\rho}}(\iota(\alpha)))$ et $\psi^*_A(j^{\rho}_r(\alpha))$.\\

En particulier, pour tout sous-espace $X$ de $\underline{E}G$ fermé et $G$-compact, on a l'égalité $$\psi_{B,*}(j_{\mathcal{B}^{\rho}}(\iota(\alpha)))=\psi^*_{C_0(X)}(j^{\rho}_r(\alpha)),$$ dans $KK^{\mathrm{ban}}(\mathcal{B}^{\rho}(G,C_0(X)),B\rG)$ pour tout $\alpha\in KK_G(C_0(X),B)$. Ceci implique alors que $\mu_{\rho,r}=\psi_{B,*}\circ\mu_{\mathcal{B}^{\rho}}$.
\end{proof}

Nous pouvons maintenant montrer le résultat principal de cette section qui est donné par le théorème suivant
\begin{Theo}\label{complétions}
 Soit $G$ un groupe localement compact satisfaisant les deux conditions suivantes:
 \begin{enumerate}
 \item il existe une complétion inconditionnelle de $C_c(G)$ qui est une sous-algèbre faiblement pleine de $C^*_r(G)$;
  \item pour toute complétion inconditionnelle $\mathcal{B}(G)$ de $C_c(G)$ le morphisme défini par Lafforgue $\mu_{\mathcal{B}}: K^{\mathrm{top}}(G)\rightarrow K(\mathcal{B}(G))$ est un isomorphisme.
      \end{enumerate}
 Soit $\rho$ une représentation de dimension finie de $G$. Alors le morphisme de Baum-Connes réduit tordu par rapport à $\rho$ $$\mu_{\rho,r}:K^{\mathrm{top}}(G)\rightarrow K(\mathcal{A}^{\rho}_r(G)),$$
est un isomorphisme.
\end{Theo}
\begin{proof}
Soit $\mathcal{B}(G)$ une complétion inconditionnelle de $C_c(G)$ stable par calcul fonctionnel holomorphe dans $C^*_r(G)$. Soit $\mathcal{B}^{\rho}(G)$ la complétion inconditionnelle de $C_c(G)$ définie comme ci-dessus. Alors, d'après le théorème \ref{ssalpondérée}, le morphisme d'algèbres de Banach $\psi:\mathcal{B}^{\rho}(G)\rightarrow\mathcal{A}^{\rho}_r(G)$ induit un isomorphisme en $K$-théorie. De plus, d'après la proposition \ref{BCcomplétions}, $\mu_{\rho,r}=\psi_*\circ\mu_{\mathcal{B}^{\rho}}$, donc $\mu_{\rho,r}$ est bien un isomorphisme.
\end{proof}
Dans \cite{Lafforgue02}, Lafforgue a démontré que les groupes appartenant à la classe $\mathcal{C}'$ vérifient la condition 2 du théorème \ref{complétions} (cf. \cite[Théorème 0.0.2]{Lafforgue02}). De plus, il a montré que si $G$ est un groupe de Lie réductif réel, une variante de l'algèbre de Schwartz généralisée (cf. \cite[Chapitre 4]{Lafforgue02}), qui est une complétion inconditionnelle de $C_c(G)$, est aussi une sous-algèbre de Banach faiblement pleine de $C^*_r(G)$. De plus, si un groupe discret $\Gamma$ a la propriété (RD), une variante de l'algèbre de Jolissaint, $H^s(\Gamma)$, qui est aussi une complétion inconditionnelle, est une sous-algèbre de Banach faiblement pleine de $C^*_r(\Gamma)$. On a alors le corollaire suivant
\begin{Cor}
 Pour toute représentation de dimension finie $\rho$, le morphisme de Baum-Connes réduit tordu $\mu_{\rho,r}$ est un isomorphisme pour les groupes suivants:
\begin{itemize}
 \item les groupes réductifs réels,
\item tous les groupes discrets appartenant à la classe $\mathcal{C}'$ et possédant la propriété (RD), donc, en particulier les sous-groupes discrets cocompacts de $Sp(n,1)$, $F_{4(-20)}$, $SL_3(F)$ où $F$ est un corps local, $SL_3(\mathbb{H})$ et $E_{6(-26)}$, et tous les groupes hyperboliques.
\end{itemize}
\end{Cor}

\section{Action sur $K^{\mathrm{top}}(G)$ par le produit tensoriel par $\rho$}
\subsection{Définitions et énoncé du théorème principal}
Étant donné un groupe localement compact $G$ et une représentation de dimension finie $(\rho,V)$ de $G$, nous allons considérer la longueur $\ell$ sur $G$ définie de la façon suivante: pour tout $g\in G$ on pose $$\ell(g):=\max\big(\log(\|\rho(g^{-1})\|_{\End(V)}),\log(\|\rho(g)\|_{\End(V)})\big).$$
 On a alors $$\|\rho(g)v\|_{V}\leq e^{\ell(g)}\|v\|_{V},$$ pour tout $v\in V$ et pour tout $g\in G$. Le couple $(V,0)$ définit alors un élément de $E^\mathrm{ban}_{G,\ell}(\CC,\CC)$ (cf. \cite[Définition 1.2.2]{Lafforgue02}). On note $[V]$ sa classe dans $KK^\mathrm{ban}_{G,\ell}(\CC,\CC)$.
\begin{Def}\label{otimesV}
Soit $A$ une $G$-$C^*$-algèbre. La représentation $(\rho,V)$ de $G$ définit un morphisme de groupes de $K(A\rG)$ dans $K(C^*_r(G,A))$. En effet, soit $C^*_r(G,A)\otimes V$ le $C^*_r(G,A)\otimes\CC$-module hilbertien construit par produit tensoriel externe. D'après la définition de $A\rG$, il est clair que l'application $$e_g\mapsto\big(h\otimes v\mapsto e_g*h\otimes\rho(g)v\big),$$
 pour tout $g\in G$, $h\in C_c(G,A)$ et $v\in V$, induit un morphisme d'algèbres de Banach $$A\rG\rightarrow \mathcal{L}_{C^*_r(G,A)}(C^*_r(G,A)\otimes V).$$ Le produit tensoriel hilbertien $C^*_r(G,A)\otimes V$ est alors un $(A\rG,C^*_r(G,A))$-bimodule de Banach (cf. \cite{Lafforgue02}), et le couple $(C^*_r(G,A)\otimes V,0)$ définit alors un élément de $E^{\mathrm{ban}}(A\rG,C^*_r(G,A))$. On note $[C^*_r(G,A)\otimes V]$ sa classe dans $KK^{\mathrm{ban}}(A\rG,C^*_r(G,A))$. Soit $$\Sigma: KK^{\mathrm{ban}}(A\rG,C^*_r(G,A))\rightarrow \mathrm{Hom}\big(K(A\rG),K(C^*_r(G,A))\big),$$ donné par l'action de la $KK$-théorie banachique sur la $K$-théorie (cf. \cite[Proposition 1.2.9]{Lafforgue02}). On a alors un morphisme 
$$\Sigma([C^*_r(G,A)\otimes V]):K(A\rG)\rightarrow K(C^*_r(G,A)),$$
que l'on note $\Lambda_{\rho,r}$.
\end{Def}
\begin{Rem}
 La représentation $(\rho,V)$ étant de dimension finie, on aurait pu définir le morphisme précédent de la manière suivante:
soit $\tau$ le morphisme d'algèbres de Banach de $A\rG$ dans $C^*_r(G,A)\otimes\mathrm{End}(V)$ induit par l'application
$$e_g\mapsto e_g\otimes\rho(g).$$
On note $\tau_*$ le morphisme de groupes induit par $\tau$ de $K(A\rG)$ dans $K\big(C^*_r(G,A)\otimes\mathrm{End}(V)\big)$. Comme $$C^*_r(G,A)\otimes\mathrm{End}(V)\simeq M_n(C^*_r(G,A))$$ où $n$ est la dimension de l'espace vectoriel $V$, alors, par équivalence de Morita, $\tau$ induit un morphisme,
$$\tau_*:K(A\rG)\rightarrow K(C^*_r(G,A)).$$
Il est facile de voir que $\tau_*$ est égal au morphisme $\Lambda_{\rho,r}$ défini par la proposition \ref{otimesV}.
\end{Rem}

\begin{Rem}
De façon analogue, $(\rho,V)$ définit un morphisme de groupes $$\Lambda_{\rho}:K(A\G)\rightarrow K(C^*(G,A)),$$
où $A\G$ est le produit croisé tordu maximal.  \\
\end{Rem}

D'autre part, une action de $R_F(G)$ sur $K^{\mathrm{top}}(G)$ est définie de façon très naturelle. Dans le cas des représentations unitaires, cette action a été définie par Kasparov. En effet, pour tout groupe localement compact $G$, l'anneau des représentations de Kasparov $KK_G(\CC,\CC)$, défini dans \cite{Kasparov88} et noté $R(G)$, est le groupe abélien formé des classes d'homotopie de représentations de Fredholm unitaires de $G$, sur lequel le produit de Kasparov définit une structure d'anneau commutatif. L'anneau des représentations unitaires de dimension finie de $G$ s'envoie sur $R(G)$; si $G$ est un groupe compact, cette application est un isomorphisme (voir par exemple \cite[Paragraph 8]{Higson98}). De plus, pour toutes $G$-$C^*$-algèbres $A$ et $B$, le produit de Kasparov définit une structure de $R(G)$-module sur le groupe $KK_G(A,B)$. En particulier, si $X$ est un $G$-espace propre (et $G$-compact), le groupe $KK_G(C_0(X),\CC)$ est muni d'une structure de module sur $R(G)$. En passant à la limite inductive, il est clair que le produit de Kasparov définit alors une structure de $R(G)$-module sur $K^{\mathrm{top}}(G)$, ce qui définit une action de l'anneau des représentations \emph{unitaires} de dimension finie de $G$ sur $K^{\mathrm{top}}(G)$. Nous allons généraliser cette action au cas des représentations non unitaires de $G$.\\ 

Supposons d'abord que $G$ soit un groupe de Lie connexe et $K$ un sous-groupe compact maximal de $G$ tels que l'espace $G/K$ soit une variété de dimension paire munie d'une structure spin$^{\CC}$ $G$-équivariante. Dans ce cas, il est clair que $K^{\mathrm{top}}(G)$ est muni d'une structure de module sur $R_F(G)$. En effet, $K^{\mathrm{top}}(G)$ est isomorphe à $R(K)$, l'anneau des représentations unitaires de $K$ et toute représentation de dimension finie $\rho$ non unitaire définit  un endomorphisme de $R(K)$: comme $K$ est un groupe compact, la restriction de $\rho$ sur $K$ est équivalente à une représentation unitaire de $K$ et donc $\rho\vert_K$ définit un élément de $R(K)$. Le produit tensoriel par $\rho\vert_K$,
\begin{align*}
R(K)&\rightarrow R(K)\\
[\sigma]&\mapsto [\rho\vert_K\otimes\sigma],
\end{align*}
induit un endomorphisme $\Upsilon_{\rho}$ de $K^{\mathrm{top}}(G)$ qui définit l'action $R_F(G)$ sur $K^{\mathrm{top}}(G)$.\\

Dans le cas général, étant donné un $G$-espace $X$ qui soit $G$-propre et $G$-compact, on considère le fibré triviale $G$-équivariant au dessus de $X$ de fibre $V$, que l'on note $\mathcal{V}$. On a alors le lemme suivant
\begin{Lemma}
 Le fibré $\mathcal{V}$ peut être muni d'une structure hermitienne $G$-équivariante.
\end{Lemma}
\begin{proof}
 Soit $b\in C_c(X,\RR^+)$ une fonction à support compact sur $X$ telle que $\int_Gb(g^{-1}x)dg=1$ pour tout $x\in X$. Soit $K$ le support de $b$ qui est une partie compacte de $X$. On munit $V$ d'une structure hermitienne quelconque au dessus de $K$ que l'on note $\langle.,.\rangle_{1,x}$ pour tout $x\in K$. Soient $x\in X$ et $v_1,v_2$ appartenant à la fibre de $\mathcal{V}$ au dessus de $x$ notée $\mathcal{V}_x $. On pose:
$$\langle v_1,v_2\rangle_{2,x}:=\int_Gb(g^{-1}x)\big\langle\rho(g^{-1})v_1,\rho(g^{-1})v_2\big\rangle_{1,g^{-1}x}dg.$$
Ceci définit une structure hermitienne $G$-équivariante au dessus de $\mathcal{V}$. En effet, pour $g_1\in G$, $x\in X$ et $v_1,v_2\in \mathcal{V}_x$
\begin{align*}
  g_1\big\langle v_1,v_2\big\rangle_{2,x}&=\big\langle \rho(g_1^{-1})v_1,\rho(g_1^{-1})v_2\big\rangle_{2,g_1^{-1}x}\\
&=\int_Gb(g^{-1}g_1^{-1}x)\big\langle\rho(g^{-1}g_1^{-1})v_1,\rho(g^{-1}g_1^{-1})v_2\big\rangle_{1,g^{-1}g_1^{-1}x}dg\\
&=\int_Gb(h^{-1}x)\big\langle\rho(h^{-1})v_1,\rho(h^{-1})v_2\big\rangle_{1,h^{-1}x}dg\\
&=\big\langle v_1,v_2\big\rangle_{2,x}.
\end{align*}
\end{proof}
On considère maintenant l'espace $C_0(X,\mathcal{V})$ des sections de $\mathcal{V}$ qui s'annulent à l'infini, où $\mathcal{V}$ est muni de la structure hermitienne $G$-équivariante, $\langle.,.\rangle_2$, définie ci-dessus.
\begin{Lemma}
\sloppy Le couple $(C_0(X,\mathcal{V}),0)$ définit un élément de $E_G(C_0(X),C_0(X))$. On note $[C_0(X,\mathcal{V})]$ sa classe dans $KK_G(C_0(X),C_0(X))$.
\end{Lemma}
\begin{proof}
En effet, $C_0(X)$ agit sur $C_0(X,\mathcal{V})$ (à gauche et à droite) par multiplication point par point. On définit un produit scalaire sur $C_0(X,\mathcal{V})$ à valeurs dans $C_0(X)$ de la façon suivante:
étant données $s_1$ et $s_2$ deux sections de $\mathcal{V}$ qui s'annulent à l'infini,
$$\langle s_1,s_2\rangle=(x\mapsto\langle s_1(x),s_2(x)\rangle_{2,x}).$$
Ce produit scalaire fait de $C_0(X,\mathcal{V})$ un $C_0(X)$-module hilbertien $G$-équivariant. L'action de $C_0(X)$ à gauche commute avec l'action à droite trivialement.\\
\end{proof}
Maintenant, si $A$ est une $G$-$C^*$-algèbre, comme $[C_0(X,\mathcal{V})]$ est un élément de $KK_G(C_0(X),C_0(X))$, le produit de Kasparov par $[C_0(X,\mathcal{V})]$ induit un morphisme de groupes
$$\xymatrix{KK_G(C_0(X),A)\ar[rrr]^{\scriptscriptstyle[C_0(X,\mathcal{V})]\otimes_{C_0(X)}}&&&KK_G(C_0(X),A)\\}.$$ En passant à la limite inductive on obtient un morphisme
$$\Upsilon_{\rho}:K^{\mathrm{top}}(G,A){\longrightarrow}K^{\mathrm{top}}(G,A).$$

Nous allons démontrer le théorème suivant
\begin{Theo}\label{principal}
 Le morphisme de groupes de $K^{\mathrm{top}}(G,A)$ dans lui-même induit par le produit de Kasparov par $[C_0(X,\mathcal{V})]$ rend commutatifs les deux diagrammes suivants
$$\xymatrix{
    K^{\mathrm{top}}(G,A)\ar[r]^{\mu^A_{\rho,r}}\ar[d]_{\Upsilon_{\rho}}& K(A\rG)\ar[d]^{\Lambda_{\rho,r}}\\
K^{\mathrm{top}}(G,A)\ar[r]^{\mu^A_r}& K(C^*_r(G,A))
}\quad\hbox{et}\quad\xymatrix{
    K^{\mathrm{top}}(G,A)\ar[r]^{\mu^A_{\rho}}\ar[d]_{\Upsilon_{\rho}}& K(A\G)\ar[d]^{\Lambda_{\rho}}\\
K^{\mathrm{top}}(G,A)\ar[r]^{\mu^A}& K(C^*(G,A)).
}$$
\end{Theo}

Nous allons d'abord montrer un résultat analogue pour les algèbres $L^1$ qui implique le théorème \ref{principal}.

\subsection{Algèbres $L^1$. Rappels et notations}\label{L^1}
Étant donnés un groupe localement compact $G$ et une $G$-$C^*$-algèbre $A$, on rappelle que l'application identité sur l'espace des fonctions continues à support compact sur $G$ à valeurs dans $A$, $C_c(G,A)$, se prolonge en un morphisme d'algèbres de Banach de $L^1(G,A)$ dans $C^*_r(G,A)$ (resp. $C^*(G,A)$), et donc $L^1(G,A)$ est une sous-algèbre dense de $C^*_r(G,A)$ (resp. $C^*(G,A)$). De même, si on note $L^{1,\rho}(G,A)$ la complétion de $C_c(G,A)$ pour la norme
 $$\|f\|_{L^{1,\rho}(G,A)}=\int_G\|f(g)\|_A\|\rho(g)\|_{\mathrm{End}(V)}dg,$$
pour $f\in C_c(G,A)$, alors l'application identité de $C_c(G,A)$ se prolonge en un morphisme d'algèbres de Banach de $L^{1,\rho}(G,A)$ dans le produit croisé tordu $A\rG$ (resp. $A\G$), et donc $L^{1,\rho}(G,A)$ est une sous-algèbre dense de $A\rG$ (resp. $A\G$). Nous allons montrer un énoncé pour les algèbres $L^1(G,A)$ et $L^{1,\rho}(G,A)$ analogue à celui du théorème \ref{principal} qui va impliquer le théorème \ref{principal}.
\begin{Rem}\label{unitaire}
 Si $\rho$ est une représentation unitaire, alors, pour toute $G$-$C^*$-algèbre $A$, $L^{1,\rho}(G,A)=L^1(G,A)$.
\end{Rem}

On rappelle que les algèbres $L^1(G)$ et $L^{1,\rho}(G)$ sont des complétions inconditionnelles de $C_c(G)$ et donc que Lafforgue a défini dans \cite[Proposition-Définition 1.5.5]{Lafforgue02}, des morphismes de descente pour ces algèbres:
\begin{align*}
j_{L^1}:KK^{\mathrm{ban}}_{G,\ell}(A,B)&\rightarrow KK^{\mathrm{ban}}(L^{1,\rho}(G,A),L^1(G,B))\\ j_{L^{1,\rho}}:KK^{\mathrm{ban}}_{G}(A,B)&\rightarrow KK^{\mathrm{ban}}(L^{1,\rho}(G,A),L^{1,\rho}(G,B)),                                \end{align*}
pour $A$ et $B$ des $G$-$C^*$-algèbres et pour $\ell$ la longueur sur $G$ définit par la norme de $\rho$. \\
Par abus de notation, pour toutes $G$-$C^*$-algèbre $A$ et $B$, nous allons définir  un troisième morphisme de ``descente``$$j_{\rho}:KK_G(A,B)\rightarrow KK^{\mathrm{ban}}(L^{1,\rho}(G,A),L^{1,\rho}(G,B)),$$
comme la composée de $\iota$ et $j_{L^{1,\rho}}$, c'est-à-dire $j_{\rho}:=j_{L^{1,\rho}}\circ\iota$. Ce morphisme est l'analogue sur $L^{1,\rho}$ du morphisme de descente tordu défini dans \cite[Définition 2.9]{Gomez08}.\\

Soit $X$ une partie $G$-compacte de $\underline{E}G$. Soit $c$ une fonction continue à support compact sur $X$ et à valeurs dans $\RR_+$ telle que, pour tout $x\in X$, $\int_Gc(g^{-1}x)dg=1$. Soit $p$ la fonction sur $G\times X$ définie par la formule $$p(g,x)=\sqrt{c(x)c(g^{-1}x)}.$$
La fonction $p$ définit alors un projecteur de $C_c(G,C_0(X))$, que l'on note $p$ par abus de notation. On note $\Delta_{\rho}$ l'élément de $K(L^{1,\rho}(G,C_0(X)))$ qu'il définit.\\
On rappelle alors que, pour toute $G$-$C^*$-algèbre $A$, la variante du morphisme de Baum-Connes à valeurs dans $K(L^{1,\rho}(G,A))$ définie par Lafforgue dans \cite[1.7]{Lafforgue02}, est donnée, à passage à la limite inductive près, par le morphisme
$$\mu^A_{L^{1,\rho}}:KK_G(C_0(X),A)\rightarrow K(L^{1,\rho}(G,A)),$$défini par la formule $$\mu_{L^{1,\rho}}(\alpha)=\Sigma(j_{\rho}(\alpha))(\Delta_{\rho}),$$ pour tout élément $\alpha$ dans $KK_G(C_0(X),A)$.\\

Pour montrer qu'un énoncé analogue à l'énoncé du théorème \ref{principal} pour les algèbres $L^1$ implique le théorème \ref{principal}, on aura besoin du lemme de compatibilité suivant

\begin{Lemma}\label{compatL1}
 Les diagrammes 
$$\xymatrix{K^{\mathrm{top}}(G,A)\ar[rd]_{\mu^A_{\rho,r}}\ar[r]^{\mu^A_{L^{1,\rho}}}& K(L^{1,\rho}(G,A))\ar[d]^{i_*}\\
&K(A\rG)
}\quad\hbox{et}\quad \xymatrix{K^{\mathrm{top}}(G,A)\ar[rd]_{\mu^A_{\rho}}\ar[r]^{\mu^A_{L^{1,\rho}}}& K(L^{1,\rho}(G,A))\ar[d]^{i'_*}\\
&K(A\G)
}$$
où $i:L^{1,\rho}(G,A)\rightarrow A\rG$ (resp. $i':L^{1,\rho}(G,A)\rightarrow A\G$) est le prolongement de l'application identité sur $C_c(G,A)$, sont commutatifs.
\end{Lemma}
\begin{proof}
 La démonstration est analogue à la démonstration de la proposition \ref{BCcomplétions}.
\end{proof}
\begin{Rem}
 Le lemme \ref{compatL1} et la remarque \ref{unitaire} impliquent en particulier que si $\rho$ est une représentation unitaire alors $\mu_{\rho,r}^A=\mu_r^A$
(resp. $\mu_{\rho}^A=\mu^A$) pour toute $G$-$C^*$-algèbre $A$ et donc le morphisme de Baum-Connes tordu coïncide avec le morphisme de Baum-Connes classique.
\end{Rem}

\subsection{Démonstration du théorème \ref{principal}}
Soit $X$ une partie $G$-compacte de $\underline{E}G$. On va montrer le théorème suivant

\begin{Theo}\label{diagX}
Pour toute $G$-$C^*$-algèbre $A$, le morphisme, noté $\Upsilon_{\rho}$, qui à un élément $\alpha\in KK_G(C_0(X),A)$ associe le produit de Kasparov $[C_0(X,\mathcal{V})]\otimes_{C_0(X)}\alpha$ dans $KK_G(C_0(X),A)$ rend commutatif le diagramme suivant
$$\xymatrix{
    KK_G(C_0(X),A)\ar[r]^{\,\,\,\,\,\,\,\,\,\mu^A_{\rho,r}}\ar[d]_{\Upsilon_{\rho}}& K(A\rG)\ar[d]^{\Lambda_{\rho,r}}\\
KK_G(C_0(X),A)\ar[r]^{\,\,\,\,\,\,\,\,\,\mu^A_r}& K(C^*_r(G,A)).\\
}$$
\end{Theo}

La preuve repose sur un résultat analogue pour les algèbres $L^1$ que nous allons énoncer plus bas et qui implique aussi la commutativité du diagramme du théorème \ref{principal} pour les produits croisés maximaux de façon analogue. Pour énoncer le résultat pour les algèbres $L^1$, nous avons besoin de la définition suivante.
\begin{Def}
Pour toute $G$-$C^*$-algèbre $A$, la représentation $(\rho,V)$ de $G$ définit le morphisme de groupes suivant,
$$\Sigma\big(j_{\mathrm{L^1}}(\sigma_A([V]))\big):K(L^{1,\rho}(G,A))\rightarrow K(L^1(G,A)),$$
où, on rappelle que, dans ce cas, $$\Sigma:KK^{\mathrm{ban}}\big(L^{1,\rho}(G,A),L^1(G,A)\big)\rightarrow \mathrm{Hom}\big(K(L^{1,\rho}(G,A)), K(L^1(G,A))\big).$$
En effet, $[V]$ est un élément de $KK^{\mathrm{ban}}_{G,\ell}(\CC,\CC)$ et donc $j_{L^1}(\sigma_A([V]))$ appartient à $KK^{\mathrm{ban}}(L^{1,\rho}(G,A),L^{1}(G,A))$.
\end{Def}

Nous allons alors montrer le théorème suivant qui est l'analogue pour les algèbres $L^1$ du théorème \ref{diagX} annoncé dans la section \ref{L^1}.
\begin{Theo}\label{lemmediagX}
Pour toute $G$-$C^*$-algèbre $A$ et pour toute partie $G$-compacte $X$ de $\underline{E}G$, le diagramme 
$$\xymatrix{
    KK_G(C_0(X),A)\ar[r]^{\,\,\,\,\,\,\,\,\,\mu^A_{L^{1,\rho}}}\ar[d]_{\Upsilon_{\rho}}& K(L^{1,\rho}(G,A))\ar[d]^{\Sigma(j_{L^1}(\sigma_A([V])))}\\
KK_G(C_0(X),A)\ar[r]^{\,\,\,\,\,\,\,\,\,\mu^A_{L^1}}& K(L^1(G,A)),\\
}$$
est commutatif.
\end{Theo}

Montrons d'abord que le théorème \ref{lemmediagX} implique le théorème \ref{diagX}.\\
Soient
 $$i_1:L^1(G,A)\rightarrow C^*_r(G,A)\quad
\hbox{et}\quad i_2:L^{1,\rho}(G,A)\rightarrow A\rG,$$
les inclusions naturelles qui prolongent l'application identité sur $C_c(G,A)$. La fonctorialité de $\Sigma$ implique que le morphisme $\Lambda_{\rho,r}$ de $K(A\rG)$ dans $K(C^*_r(G,A))$ rend commutatif le diagramme suivant:
$$\xymatrix{
    K(L^{1,\rho}(G,A))\ar[r]^{i_{2,*}}\ar[d]_{\Sigma(j_{L^1}(\sigma_A([V])))}& K(A\rG)\ar[d]^{\Lambda_{\rho,r}}\\
K(L^1(G,A))\ar[r]^{i_{1,*}}& K(C^*_r(G,A)).
}$$
En effet, il est facile de voir que $$i_{2}^*([C^*_r(G,A)\otimes V])=i_{1,*}\big(j_{L¹}(\sigma_A([V]))\big),$$
dans $KK^{\mathrm{ban}}(L^{1,\rho}(G,A),C^*_r(G,A))$. D'où,
\begin{align*}
 \Sigma([C^*_r(G,A)\otimes V])i_{2,*}&=i_2^*\big(\Sigma([C^*_r(G,A)\otimes V)]\big),\\
&=\Sigma(i_2^*([C^*_r(G,A)\otimes V])),\\
&=\Sigma(i_{1,*}\big(j_{L^1}(\sigma_A([V]))\big),\\
&=i_{1,*}\big(\Sigma(j_{L^1}(\sigma_A([V])))\big),
\end{align*}
et comme $\Lambda_{\rho,r}=\Sigma([C^*_r(G)\otimes V])$ par définition, alors le diagramme est commutatif ce qui prouve que le théorème \ref{lemmediagX} implique le théorème \ref{diagX}. Il en est de même dans le cas des produits croisés maximaux.\\

La démonstration du théorème \ref{lemmediagX} repose sur le lemme et la proposition suivants

\begin{Lemma}\label{action}
Les éléments $\iota([C_0(X,\mathcal{V})])$ et $\sigma_{C_0(X)}([V])$ sont égaux dans $KK^{\mathrm{ban}}_{G,\ell}(C_0(X),C_0(X))$.
\end{Lemma}
Ce lemme implique en particulier que les éléments $j_{L^{1}}(\iota([C_0(X,\mathcal{V})]))$ et $j_{L^{1}}(\sigma_{C_0(X)}([V]))$ dans $KK^{\mathrm{ban}}(L^{1,\rho}(G,C_0(X)),L^1(G,C_0(X)))$ sont égaux et donc qu'ils agissent de la même manière sur l'élément $\Delta_{\rho}$ de $K(L^{1,\rho}(G,C_0(X)))$, c'est-à-dire
$$\Sigma(j_{L^{1}}(\iota([C_0(X,\mathcal{V})])))(\Delta_{\rho})=\Sigma(j_{L^{1}}(\sigma_{C_0(X)}([V])))(\Delta_{\rho}).$$
\begin{proof}
Nous allons montrer que $\iota([C_0(X,\mathcal{V})])$ et $\sigma_{C_0(X)}([V])$ sont homotopes dans $E^{\mathrm{ban}}_{G,\ell}(C_0(X),C_0(X))$.\\
Soit $s$ une fonction sur $X$ à valeurs dans $V$ qui s'annule à l'infini. On pose:
 \begin{align*}
\|s\|_0&:=\|s\|_{C_0(X,V)}=\sup\limits_{x\in X}\|s(x)\|_V\\
\|s\|_1&:=\|s\|_{C_0(X,\mathcal{V})}=\sup\limits_{x\in X}\|s(x)\|_{\mathcal{V}}.
 \end{align*}
Soit $\mathcal{H}_t$, pour $t\in [0,1]$, l'espace de Hilbert des fonctions sur $X$ à valeurs dans $V$ qui s'annulent à l'infini pour la norme:
$$\|s\|_t=t\|s\|_1+(1-t)\|s\|_0.$$
Alors, $\mathcal{H}=(\mathcal{H}_t)_{t\in[0,1]}\in E^{\mathrm{ban}}_{G,\ell}\big(C_0(X),C_0(X)[0,1]\big)$. Il est clair que $\mathcal{H}$ réalise l'homotopie cherchée.
\end{proof}

\begin{Pro}\label{diag}
Soient $A$ et $B$ des $C^*$-algèbres et soit $\alpha$ un élément de $KK_G(A,B)$. Alors le diagramme 
$$\xymatrix{
    K(L^{1,\rho}(G,A))\ar[rr]^{\Sigma(j_{\rho}(\alpha))}\ar[dd]^{\Sigma(j_{L^1}(\sigma_A([V])))}& &K(L^{1,\rho}(G,B))\ar[dd]^{\Sigma(j_{L^1}(\sigma_B([V])))}\\
&&\\
K(L^{1}(G,A))\ar[rr]^{\Sigma(j_{L^{1}}(\alpha))}& &K(L^{1}(G,B))\\
}$$
 est commutatif.
\end{Pro}
\begin{proof}
 Soit $\alpha\in KK_G(A,B)$. D'après le lemme 1.6.11 de \cite{Lafforgue02}, il existe une $G$-$C^*$-algèbre que l'on note $A_1$, deux morphismes $G$-équivariants $\theta:A_1\rightarrow A$ et $\eta:A_1\rightarrow B$ et un élément $\alpha_1$ dans $ KK_G(A,A_1)$, tels que $\theta^*(\alpha_1)=\mathrm{Id}_{A_1}$ dans $KK_G(A_1,A_1)$, $\theta_*(\alpha_1)=\mathrm{Id}_{A}$ dans $KK_G(A,A)$, et $\theta^*(\alpha)=[\eta]$ dans $KK_G(A,B)$. On peut écrire alors le diagramme suivant en $KK$-théorie\\
$$\xymatrix@1{
&&A_1\ar[lld]_{\theta}\ar[rrd]^{\eta}\ar@{-->}[dd]^<(0.6){\sigma_{A_1}([V])}&&\\
    A\ar@{-->}[rrrr]^<(0.3){\alpha}\ar@{-->}[dd]_{\sigma_A([V])}\ar@/^2pc/ @{-->}[rru]^(0.6){\alpha_1}&&& &B\ar@{-->}[dd]^{\sigma_B([V])}\\
&&A_1\ar[lld]_{\theta}\ar[rrd]^{\eta} &&\\
A\ar@{-->}[rrrr]^{\alpha}\ar@/^2pc/ @{-->}[rru]^(0.6){\alpha_1}&&& &B\\
}$$\\

\noindent où les flèches en pointillés désignent des éléments en $KK$-théorie qui ne sont pas nécessairement donnés par des morphismes d'algèbres.\\

\begin{Lemma}\label{diagKK}
Comme $[V]\in KK_{G,l}^{\mathrm{ban}}(\CC,\CC)$, on a alors les deux égalités suivantes
$$\theta^*(\sigma_A([V]))=\theta_*(\sigma_{A_1}([V]))\quad\hbox{et}\quad\eta^*(\sigma_B([V]))=\eta_*(\sigma_{A_1}([V])),$$
dans $KK^{\mathrm{ban}}_{G,\ell}(A_1,A)$ et $KK^{\mathrm{ban}}_{G,\ell}(A_1,B)$.
\end{Lemma}
\begin{proof}
 Le lemme découle de la fonctorialité de $KK^{\mathrm{ban}}$.
\end{proof}

On rappelle que nous avons noté $j_{\rho}=j_{L^{1,\rho}}\circ\iota$ pour simplifier les notations. Le fait que $\Sigma$, $j_{L^1}$ et $j_{L^{1,\rho}}$ soient fonctoriels nous permet, d'une part, d'écrire le diagramme suivant:\\
$$\xymatrix{
&K(L^{1,\rho}(G,A_1))\ar[ld]_{j_{\rho}(\theta)_*}\ar[rd]|{\quad}^{j_{\rho}(\eta)_*}\ar[ddd]|-(0.35){\phantom{hoooola}}^<(0.6){\Sigma(j_{L^1}(\sigma_{A_1}[V]))}&\\
    K(L^{1,\rho}(G,A))\ar[rr]^{\Sigma(j_{\rho}(\alpha))\,\,\,\,\,\,\,\,\,\,\,\,\,\,\,\,\,\,\,\,\,\,\,\,\,\,\,\,\,\,\,\,\,\,\,\,\,\,\,\,}\ar[ddd]_{\Sigma(j_{L^1}(\sigma_A([V])))}& &K(L^{1,\rho}(G,B))\ar[ddd]^{\Sigma(j_{L^1}(\sigma_B([V])))}\\
&&\\
&K(L^1(G,A_1))\ar[ld]|{\,\,\,\,\,\,\,\,\,}_{j_{L^1}(\iota(\theta))_*}\ar[rd]^{j_{L^1}(\iota(\eta))_*} &\\
K(L^{1}(G,A))\ar[rr]^{\Sigma(j_{L^{1}}(\iota(\alpha)))}& &K(L^{1}(G,B))\\
}$$\\
et d'autre part de montrer que $j_{L^1}(\iota(\theta))_*$ et $j_{\rho}(\theta)_*$ sont inversibles (voir démonstration du Lemme 1.6.11 de \cite{Lafforgue02}). On va montrer que ce diagramme est commutatif en le découpant en morceaux.
\begin{Lemma}\label{triangle}
On a les égalités suivantes:
 $$j_{L^1}(\iota(\eta))_*\circ j_{L^1}(\iota(\theta))_*^{-1}=\Sigma\big(j_{L^1}(\iota(\alpha))\big),$$
et
$$j_{\rho}(\eta)_*\circ j_{\rho}(\theta)_*^{-1}=\Sigma\big(j_{\rho}(\alpha)\big).$$
\end{Lemma}
\begin{proof}
 Par fonctorialité et par définition de $\eta$ et de $\theta$:
\begin{align*}
 \Sigma(j_{L^1}(\iota(\alpha)))\circ j_{L^1}(\iota(\theta))_*&=\Sigma\big(j_{L^1}(\iota(\theta))^{*}(j_{L^{1}}(\iota(\alpha)))\big)\\
 &=\Sigma\big(j_{L^1}(\iota(\theta^*(\alpha)))\big),\\
&=\Sigma\big(j_{L^1}(\iota(\eta))\big),\\
&=j_{L^1}(\iota(\eta))_*.
\end{align*}
La deuxième égalité se démontre de façon complètement analogue.
\end{proof}
\begin{Lemma}\label{rombo}
 On a les égalités suivantes:
$$\Sigma\big(j_{L^1}(\sigma_{\scriptscriptstyle A}([V]))\big)\circ j_{L^{1,\rho}}(\iota(\theta))_*=j_{L^1}(\iota(\theta))_*\circ\Sigma\big(j_{L^1}(\sigma_{\scriptscriptstyle A_1}([V]))\big),$$
et
$$\Sigma\big(j_{L^1}(\sigma_{\scriptscriptstyle B}([V]))\big)\circ j_{L^{1,\rho}}(\iota(\eta))_*=j_{L^1}(\iota(\eta))_*\circ\Sigma\big(j_{L^1}(\sigma_{\scriptscriptstyle A_1}([V]))\big).$$

\end{Lemma}

\begin{proof}
La fonctorialité et le lemme \ref{diagKK} permettent d'avoir les égalités suivantes:
\begin{align*}
 \Sigma\big(j_{L^1}(\sigma_{\scriptscriptstyle A}([V]))\big)\circ j_{L^{1,\rho}}(\iota(\theta))_*&=\Sigma\big(j_{L^1}(\theta^*(\sigma_{\scriptscriptstyle A}([V])))\big),\\
&=\Sigma\big(j_{L^1}(\theta_*(\sigma_{\scriptscriptstyle A_1}([V])))\big),\\
&=j_{L^1}(\iota(\theta))_*\circ\Sigma\big(j_{L^1}(\sigma_{\scriptscriptstyle A_1}([V]))\big),
\end{align*}
et d'autre part:
\begin{align*}
j_{L^1}(\iota(\eta))_*\circ\Sigma\big(j_{L^1}(\sigma_{A_1}([V]))\big),
&=\Sigma\big(j_{L^1}(\eta_*(\sigma_{\scriptstyle A_1}([V])))\big),\\
&=\Sigma\big(j_{L^1}(\eta^*(\sigma_B([V])))\big),\\
&=\Sigma\big(j_{L^1}(\sigma_B([V]))\big)\circ j_{L^1}(\iota(\eta))_*.
\end{align*}
\end{proof}

Maintenant on est prêt pour conclure. D'après les deux lemmes précédents, on a:
\begin{align*}
 \Sigma(j_{L^1}(\iota(\alpha)))&\circ\Sigma\big(j_{L^1}(\sigma_{\scriptscriptstyle A}([V]))\big)\circ j_{L^{1,\rho}}(\iota(\theta))_*,\\
&=\Sigma(j_{L^1}(\iota(\alpha)))\circ j_{L^1}(\iota(\theta))_*\circ\Sigma\big(j_{L^1}(\sigma_{\scriptscriptstyle A_1}([V]))\big),
\end{align*}
et ceci implique donc que:
\begin{align*}
 \Sigma\big(j_{L^1}(\iota(\alpha))\big)\circ\Sigma\big(j_{L^1}(&\sigma_{\scriptscriptstyle A}([V]))\big),\\
&=j_{L^1}(\iota(\eta))_*\circ\Sigma\big(j_{L^1}(\sigma_{A_1}([V]))\big)\circ j_{L^{1,\rho}}(\iota(\theta))_*^{-1},\\
&=\Sigma\big(j_{L^1}(\sigma_B([V]))\big)\circ j_{L^{1\rho}}(\iota(\eta))_*\circ j_{L^{1,\rho}}(\iota(\theta))_*^{-1},\\
&=\Sigma\big(j_{L^1}\sigma_B([V]))\big)\circ\Sigma\big(j_{L^{1,\rho}}(\iota(\alpha))\big).
\end{align*}
Et ceci termine la démonstration de la proposition \ref{diag}.
\end{proof}

\begin{proof}[Démonstration du théorème \ref{lemmediagX}]
On va maintenant montrer le théorème \ref{lemmediagX}. Nous allons noter $[\mathcal{V}]:=[C_0(X,\mathcal{V})]$ pour simplifier les notations. Soit $\alpha$ un élément de $KK_G(C_0(X),A)$.
La compatibilité de $\Sigma$ avec le produit de Kasparov (voir \cite[Proposition 1.6.10]{Lafforgue02}) implique les égalités suivantes
\begin{align*}
\mu_{L^1}([\mathcal{V}]\otimes\alpha)&=\Sigma\big(j_{L^1}(\iota([\mathcal{V}]\otimes\alpha))\big)(\Delta),\\
&=\Sigma\big(j_{L^1}(\iota(\alpha))\big)\circ\Sigma\big(j_{L^1}(\iota([\mathcal{V}]))\big)(\Delta),
\end{align*}
où $\Delta$ est l'élément de $K\big(L^1(G,C_0(X))\big)$ défini par $p$.\\ 
Mais d'après le lemme \ref{action} on a  $$\Sigma\big(j_{L^1}(\iota(\alpha))\big)\circ\Sigma\big(j_{L^1}(\iota([\mathcal{V}]))\big)(\Delta)=\Sigma\big(j_{L^1}(\iota(\alpha))\big)\circ\Sigma\big(j_{L^1}(\sigma_{C_0(X)}[V])\big)(\Delta),$$et la proposition \ref{diag} appliquée à $C_0(X)$ et à $A$ implique que le dernier membre de cette égalité est égal à $$\Sigma\big(j_{L^1}(\sigma_A([V]))\big)\circ\Sigma\big(j_{\rho}(\alpha)\big)(\Delta).$$ On a alors,
\begin{align*}
\mu_{L^1}([\mathcal{V}]\otimes\alpha)&=\Sigma\big(j_{L^1}(\iota(\alpha))\big)\circ\Sigma\big(j_{L^1}(\sigma_{C_0(X)}[V])\big)(\Delta),\\
&=\Sigma\big(j_{L^1}(\sigma_A([V]))\big)\circ\Sigma\big(j_{\rho}(\alpha)\big)(\Delta),\\
&=\Sigma\big(j_{L^1}(\sigma_A([V]))\big)\circ\mu_{L^{1,\rho}}(\alpha),
 \end{align*}
et c'est exactement ce qu'on voulait démontrer.\\
Dans la page suivante, nous écrivons un diagramme récapitulatif de la démonstration.\\
\newpage
\thispagestyle{empty}
\begin{sideways}
$$\xymatrix{
  KK_G(C_0(X),A)\ar[dddddddd]^{\Upsilon_{\rho}}\ar[rrr]^{\Sigma\circ j_{\rho}}\ar[ddddr]^{\Sigma\circ j_{L^{1}}\circ\iota}
  &&&\mathrm{Hom}\Big(K\big(L^{1,\rho}(G,C_0(X))\big),K\big(L^{1,\rho}(G,A)\big)\Big)\ar[r]^<(0.1){\scriptscriptstyle(.)(\Delta_{\rho})}\ar[dddd]^{\Sigma(j_{L^{1}}(\sigma_A[V]))}
    &K(L^{1,\rho}(G,A))\ar[dddddddd]_<(0.3){\Sigma(j_{L^1}(\sigma_A[V]))}\\
  &&&&\\
  &&&&\\
  &&&&\\
  &\mathrm{Hom}\Big(K\big(L^{1}(G,C_0(X))\big),K\big(L^{1}(G,A)\big)\Big)\ar[dddd]^{\Sigma(j_{L^{1}}(\iota[\mathcal{V}]))}\ar[rr]^{\Sigma(j_{L^{1}}(\sigma_{C_0(X)}[V]))}
  &&\mathrm{Hom}\Big(K\big(L^{1,\rho}(G,C_0(X))\big),K\big(L^{1}(G,A)\big)\Big)\ar[rdddd]^{(.)(\Delta_{\rho})}& \\
  &&&&\\
  &&&&\\
  &&&&\\
  KK_G(C_0(X),A)\ar[r]^(0.3){\scriptscriptstyle\Sigma\circ j_{ L^{1}}\circ\iota}&\mathrm{Hom}\Big(K\big(L^{1}(G,C_0(X))\big),K\big(L^{1}(G,A)\big)\Big)\ar[rrr]^{(.)(\Delta)}&&&K(L^1(G,A))\\
}$$\\
\end{sideways}
\end{proof}



\bibliographystyle{amsalpha}
\bibliography{bibliographie}

\def\cprime{$'$} \def\cprime{$'$}
\providecommand{\bysame}{\leavevmode\hbox to3em{\hrulefill}\thinspace}
\providecommand{\MR}{\relax\ifhmode\unskip\space\fi MR }
\providecommand{\MRhref}[2]{%
  \href{http://www.ams.org/mathscinet-getitem?mr=#1}{#2}
}
\providecommand{\href}[2]{#2}
\begin{thebibliography}{dlHV89}

\bibitem[BCH94]{Baum-Connes-Higson}
P.~Baum, A.~Connes, and N.~Higson, \emph{Classifying space for proper actions
  and {$K$}-theory of group {$C\sp \ast$}-algebras}, $C\sp \ast$-algebras:
  1943--1993 (San Antonio, TX, 1993), Contemp. Math., vol. 167, Amer. Math.
  Soc., Providence, RI, 1994, pp.~240--291.

\bibitem[Bos90]{Bost90}
J.~B. Bost, \emph{Principe d'{O}ka, {$K$}-th\'eorie et systèmes dynamiques non
  commutatifs}, Invent. Math. \textbf{101} (1990), 261--333.

\bibitem[Cha03]{Chatterji03}
I.~Chatterji, \emph{Property ({RD}) for cocompact lattices in a finite product
  of rank one {L}ie groups with some rank two {L}ie groups}, Geom. Dedicata
  \textbf{96} (2003), 161--177.

\bibitem[dlHV89]{delaHarpe-Valette}
P.~de~la Harpe and A.~Valette, \emph{La propri\'et\'e {$(T)$} de {K}azhdan pour
  les groupes localement compacts (avec un appendice de {M}arc {B}urger)},
  Ast\'erisque (1989), no.~175, 158, With an appendix by M. Burger.

\bibitem[GA07a]{GomezThese}
M.~P. Gomez-Aparicio, \emph{Propriéte {$(T)$} et morphisme de {B}aum-{C}onnes
  tordus par une représentation non-unitaire}, Ph.D. thesis, Université de
  Paris VII, décembre 2007.

\bibitem[GA07b]{Gomez07}
\bysame, \emph{Sur la propriété {$(T)$} tordue par un produit tensoriel}, J.
  Lie Theory \textbf{17} (2007), 505--524.

\bibitem[GA08]{Gomez08}
\bysame, \emph{Morphisme de {B}aum-{C}onnes tordu par une représentation
  non-unitaire}, preprint arXiv:0803.2472v1, 2008.

\bibitem[Hig98]{Higson98}
N.~Higson, \emph{The {B}aum-{C}onnes conjecture}, Proceedings of the
  International Congress of Mathematicians, Vol. II (Berlin, 1998), no. Extra
  Vol. II, 1998, pp.~637--646 (electronic).

\bibitem[Jol90]{Jolissaint}
P.~Jolissaint, \emph{Rapidly decreasing functions in reduced {$C\sp
  *$}-algebras of groups}, Trans. Amer. Math. Soc. \textbf{317} (1990), no.~1,
  167--196.

\bibitem[Kas88]{Kasparov88}
G.~G. Kasparov, \emph{Equivariant {$KK$}-theory and the {N}ovikov conjecture},
  Invent. Math. \textbf{91} (1988), no.~1, 147--201.

\bibitem[Kaz67]{Kazhdan}
D.~A. Kazhdan, \emph{On the connection of the dual space of a group with the
  structure of its closed subgroups}, Funkcional. Anal. i Prilo\v zen.
  \textbf{1} (1967), 71--74.

\bibitem[KV95]{Knapp-Vogan}
A.~W. Knapp and D.~A. Vogan, Jr., \emph{Cohomological induction and unitary
  representations}, Princeton Mathematical Series, vol.~45, Princeton
  University Press, Princeton, NJ, 1995.

\bibitem[Laf00]{Lafforgue00}
V.~Lafforgue, \emph{A proof of property ({RD}) for cocompact lattices of {${\rm
  SL}(3,\bold R)$} and {${\rm SL}(3,\bold C)$}}, J. Lie Theory \textbf{10}
  (2000), no.~2, 255--267.

\bibitem[Laf02]{Lafforgue02}
\bysame, \emph{{$K$}-th\'eorie bivariante pour les alg\`ebres de {B}anach et
  conjecture de {B}aum-{C}onnes}, Invent. Math. \textbf{149} (2002), no.~1,
  1--95.

\bibitem[Tu99]{Tu99}
J.~L. Tu, \emph{La conjecture de {N}ovikov pour les feuilletages
  hyperboliques}, $K$-Theory \textbf{16} (1999), no.~2, 129--184.

\bibitem[Val88]{Valette88}
Alain Valette, \emph{Le r\^ole des fibr\'es de rang fini en th\'eorie de
  {K}asparov \'equivariante}, Acad. Roy. Belg. Cl. Sci. M\'em. Collect. 8$\sp
  {\rm o}$ (2) \textbf{45} (1988), no.~6.

\bibitem[Zuc77]{Zuckerman77}
G.~Zuckerman, \emph{Tensor products of finite and infinite dimensional
  representations of semisimple {L}ie groups}, Ann. Math. (2) \textbf{106}
  (1977), no.~2, 295--308.

\end{thebibliography}

{\sc Maria Paula Gomez-Aparicio : Institut de Mathématiques de Jussieu, Equipe d'algèbres d'Opérateurs,
175 rue du chevaleret, 75013 Paris, France.}

{\it E-mail address : }\texttt{gomez@math.jussieu.fr}

\end{document}